\documentclass[review]{article}

\usepackage[english]{babel}
\usepackage[colorlinks=true,linkcolor=blue,citecolor=red,urlcolor=cyan]{hyperref}\usepackage{graphicx}
\usepackage{hyperref}
\usepackage{amsmath,amsthm}
\usepackage{amssymb}
\usepackage{enumitem}
\usepackage{geometry}
\usepackage{authblk}
\newtheorem{thm}{Theorem}
\title{Spectrum of the Dirac operator in shrinking tubes}
\date{}
\author{Nour Kerraoui}
\affil[1]{Aix-Marseille Universit\'e, CNRS, I2M, Marseille, France.}
\theoremstyle{remark}

\newtheorem{lem}{Lemma}
\theoremstyle{definition}
\newtheorem{defe}{Definition}
\newtheorem{prop}{Proposition}

\usepackage{bbm}
\newenvironment{myproof}[2] {{\textit{Proof of {#1} {#2}.}}}{\hfill$\square$}

\begin{document}
\maketitle
\begin{abstract}
In this paper we investigate the spectrum of the Dirac operator posed in a tubular neighborhood of a planar loop with infinite mass boundary conditions. We show that when the width of the tubular neighborhood goes to zero the asymptotic expansion of the eigenvalues is driven by a one dimensional operator of geometric nature involving the curvature of the loop.
\end{abstract}

\section{Introduction}
\subsection{Motivations}

The Dirichlet Laplacian in thin curved tubes was the subject of many works within the last years. It is well known that the limit operator is a Shr\"{o}dinger operator with a geometric electric potential. See for instance \cite{exn,oli,sed} where the convergence is understood in the norm resolvent sense, see also the works \cite{haa,lam,was} where the adiabatic limit in neighborhoods of Riemaniann submanifolds is investigated. Recently, the relativistic counterpart of such systems involving the Dirac operator attracted a lot of attention. One motivation among others of these works is that quasiparticles in promising materials such as graphene can mimic the quantum electrodynamics \cite{cas,FW12}.\\ The study of Dirac operators in domains was started in $1987$ by Berry and Mondragon \cite {mon} where the operator was subject to some boundary conditions called infinite mass boundary conditions. The latter is obtained by coupling the standard Dirac operator with a mass term being zero inside the domain under consideration and infinite elsewhere, this study was mathematically justified in \cite{sto}. It initiated works concerning the self-adjointness of such operators in graphene quantum dots \cite{BFVdBS,bar,loi} and its three dimensional equivalent, the 
  MIT bag model \cite{azz}.
  \\ The main motivation of the present paper is the work realized in \cite{th}, where the authors provided a spectral study of the Dirac operator in a shrinking relativistic quantum waveguide. Contrary to the non-relativistic setting, they got that the effective operator is just the free Dirac operator on the underlying curve, the geometry is expected to appear at the following order in the asymptotic expansion. In our paper we consider the Dirac operator posed in a tubular neighborhood of a planar loop and we investigate the spectrum of this operator when the domain shrinks to the reference curve. We give an expansion of the eigenvalues in the thin-width regime where the splitting of the eigenvalues is given by a one dimensional operator. The latter is defined {\it via} its quadratic form and involves the curvature. We use a standard strategy  based on perturbation theory applied to the square of this Dirac operator and make use of the min-max principle.

\subsection{Statement of the problem \& main result}
For $\ell > 0$, we set $\mathbb{T} = \mathbb{R} / \ell \mathbb{Z}$. Let $\gamma: \mathbb{T} \to \mathbb{R}^2$ be an injective arc length parametrization of a regular planar loop of class $C^4$ and set $\Gamma = \gamma(\mathbb{T})$. We are interested in the Dirac operator with infinite mass boundary conditions posed in a tubular neighborhood of size $\varepsilon > 0$ of $\Gamma$. To define this tubular neighborhood, we consider the map
\begin{equation}
	\Phi_\varepsilon : \left\{ \begin{array}{lcl}\mathbb{T} \times (-1,1) &\to& \mathbb{R}^2\\ (s,t) & \mapsto& \gamma(s) + \varepsilon t \nu(s),
	\end{array}
				\right.
\label{eqn:diffeo}
\end{equation}
where for $s\in \mathbb{T}$, $\nu(s)$ denotes the normal at the point $\gamma(s) \in \Gamma$ such that $(\gamma'(s),\nu(s))$ is a positively-oriented orthonormal basis of $\mathbb{R}^2$. A well known result of differential geometry shows the existence of $\varepsilon_0>0$ such that for all $\varepsilon \in (0,\varepsilon_0)$ the map $\Phi_\varepsilon$ is a global diffeomorphism from $\mathbb{T}\times (-1,1)$ to the tubular neighborhood of $\Gamma$ of width $2\varepsilon$ defined by $\Omega_\varepsilon := \Phi_\varepsilon\big(\mathbb{T} \times (-1,1)\big)$. As we are interested in the regime $\varepsilon \to 0$ there is no restriction to assume that $\varepsilon \in (0,\varepsilon_0)$ and we will do so all along this paper.

On its part, for $m\in \mathbb{R}$, the Dirac operator with infinite mass boundary conditions posed in $\Omega_\varepsilon$ is defined as
\begin{eqnarray}
	\begin{aligned}
		\text{dom}(\mathcal D_{\Gamma}(\varepsilon))&:=\left\lbrace u\in H^1(\Omega_{\varepsilon},\mathbb{C}^2):-i\sigma_3\sigma \cdot\nu_{\varepsilon}u=u  \,\mathrm{ on }\, \partial\Omega_{\epsilon}  \right\rbrace,\\   \mathcal D_{\Gamma}(\varepsilon)u&:=-i\sigma\cdot\nabla u+m\sigma_3u, \end{aligned}\label{bb}
\end{eqnarray}
where $\sigma_k, k\in\left\lbrace 1,2,3 \right\rbrace $ are the Pauli matrices defined by
\[
	\sigma_1 := \begin{pmatrix}0 & 1 \\ 1 & 0\end{pmatrix},\quad \sigma_2 := \begin{pmatrix}0 & -i \\ i & 0\end{pmatrix},\quad \sigma_3 := \begin{pmatrix}1 & 0 \\ 0 & -1\end{pmatrix}
\]
and for $x = (x_1,x_2) \in \mathbb{C}^2$, there holds $\sigma\cdot x = x_1 \sigma_1 + x_2 \sigma_2$. Here $\nu_\varepsilon$ denotes the outward pointing normal vector field of $\Omega_\varepsilon$.

As defined, $\mathcal{D}_\Gamma(\varepsilon)$ is known to be self-adjoint \cite[Theorem 1.1]{BFVdBS}. It also has compact resolvent, this follows from the compact embedding of $\text{dom}(D_{\Gamma}(\varepsilon))$ in $L^2(\Omega_{\varepsilon},\mathbb{C}^2)$ and it implies that its spectrum is composed of discrete eigenvalues accumulating to $\pm \infty$. Moreover, they are symmetric with respect to zero due to the invariance of the system under
charge conjugation, corresponding to the operator
$\sigma_1 C$,
where $C$ is the complex conjugation operator. Indeed, if $u$ is an eigenfunction associated with an eigenvalue $E$, one notices that $\sigma_1 C u$ is an eigenfunction associated with an eigenvalue $-E$. All along the paper $\mathbb{N}=\left\lbrace 1,2,3...\right\rbrace $ denotes the set of positive natural integer. For $j\in \mathbb{N}$, we denote by $E_j(\varepsilon)$ the $j$-th positive eigenvalue of $D_\Gamma(\varepsilon)$ counted with multiplicities.

To state the main result of this paper, we need to introduce an auxiliary one-dimensional operator defined {\it via} its quadratic form. Namely, consider the quadratic form
\begin{equation}
q^{1D}[f] = \int_{\mathbb{T}}\Big(|f' + i\kappa(\frac12-\frac1\pi) \sigma_3 f|^2 - \frac{\kappa^2}{\pi^2}|f|^2\Big) ds,\quad \text{dom}(q^{1D}) = H^1(\mathbb{T},\mathbb{C}^2).
\label{eqn:fqlon}
\end{equation}
where for $s \in \mathbb{T}$, $\kappa(s)$ is the curvature of the curve $\Gamma$ at the point $\gamma(s)$. It is defined through the Frenet formula $\gamma''(s) = \kappa(s)\nu(s)$. This quadratic form is densely defined, closed and bounded from below thus by Kato's first representation theorem \cite[Ch. VI, Thm. 2.1]{kat} it is associated with a self-adjoint operator and this operator has compact resolvent.

In order to state our main theorem we need to recall the min-max priciple.

\begin{defe} 
	Let $Q$ be a closed semi-bounded from below sesquilinear form with dense domain $\text{dom}(Q)$ in a complex Hilbert space $\mathcal{H}.$
	Let $A$ be the unique self-adjoint operator acting on $\mathcal{H}$ generated by the sesquilinear form $Q$. We define the $j$-th min-max value of $A$ by 
	\begin{align} \mu_j(A)=\inf_{
			W\subset \text{dom}(Q)\atop
			\text{dim}W=n 
		}\sup_{u\in W\setminus\left\lbrace0 \right\rbrace } \frac{Q(u,u)}{\parallel u \parallel^2_{\mathcal{H}}}.
	\end{align}
	
\end{defe}

\begin{prop}[min-max principle] 
	Let $Q$ be a closed semi-bounded from below sesquilinear form with
	dense domain in a Hilbert space $\mathcal{H}$ and let $A$ be the unique self-adjoint operator associated with $Q$. If we assume that $A$ has a compact resolvent then $\mu_j(A)$ is the $j$-th eigenvalue of $A$ counted with multiplicities.
\end{prop}

The main result of this paper reads as follows.

\begin{thm} Let $j \in \mathbb{N}$, there exists $\varepsilon_1 > 0$ such that for all  $\varepsilon \in (0,\varepsilon_1)$ there holds
\[
	E_j(\varepsilon) = \frac{\pi}{4\varepsilon} + \frac2\pi m - \frac{16}{\pi^3}m^2\varepsilon + \frac{2}\pi(m^2 + \mu_{2j}(q^{1D}))\varepsilon + \mathcal{O}(\varepsilon^2).
\]
\label{thm:1}
\end{thm}
The proof of Theorem \ref{thm:1} is performed working with the quadratic form associated with the square of the operator $\mathcal{D}_\Gamma(\varepsilon)$. First, we rewrite the problem in tubular coordinates using the map $\Phi_\varepsilon$ introduced in \eqref{eqn:diffeo}. The obtained quadratic form is posed on a weighted $L^2$-space that we modify to obtain a flat metric. Roughly speaking, the resulting quadratic form is tensored up to higher order terms. Then, we use a standard reduction dimension argument by finding an upper and a lower bound for the eigenvalues of the square of $\mathcal{D}_\Gamma(\varepsilon)$ using the min-max principle. The main idea is that the first order term in the expansion is given by a transverse operator which can be explicitly analyzed and the constant order term involves a longitudinal operator defined via the quadratic form $q^{1D}$
 introduced in \eqref{eqn:fqlon}.

\subsection{Structure of the paper}
In section $2$ we give some spectral properties of a one dimensional transverse Dirac operator.\\
Section $3$ contains some auxiliary quadratic forms that we need to study the square of the operator $\mathcal{D}_{\Gamma}(\varepsilon)$.
The last section is devoted to prove Theorem $\ref{thm:1}$ by giving an upper and a lower bound for the eigenvalues of the square of the operator.

	\section{The transverse operator}
	In this section we investigate a one dimensional transverse operator wich appears in the study of our main operator.

	Let $\delta\geq0$ and  $x = (x_1,x_2)^\top \in \mathbb{S}^1$, we consider the one-dimensional transverse operator $\mathcal{T}_x(\delta)$ defined as follows
	\[
		\mathcal{T}_x(\delta) = -i (\sigma\cdot x) \frac{d}{dt} + \delta\sigma_3,\quad \text{dom}(\mathcal{T}_x(\delta)) = \{f \in H^1((-1,1),\mathbb{C}^2) : f_2(\pm1) = \pm i (x_1+ix_2) f_1(\pm1)\}.
	\]
The main result of this paragraph is the next proposition.

\begin{prop}\label{tran} The following holds.

\begin{enumerate}[label=(\roman*)]
	\item $\mathcal{T}_x(\delta)$ is self-adjoint, has compact resolvent, its spectrum is symmetric with respect to zero and does not depend on $x$. Moreover, if we denote by $\lambda_j(\mathcal{T}_x(\delta))$ the $j$-th positive eigenvalue of $\mathcal{T}_x(\delta)$ there holds
	\begin{equation}
	(2j-1)^2\frac{\pi^2}{16} \leq \lambda_j(\mathcal{T}_x(\delta))^2-\delta^2\leq j^2\frac{\pi^2}4.
	\label{eqn:encalj}
	\end{equation}
	If $\phi_{j,\delta}^+$ is a normalized eigenfunction associated with $\lambda_j(\mathcal{T}_x(\delta))$, then $\phi_{j,\delta}^- := \sigma_1 \overline{\phi_{j,\delta}^+}$ is a normalized eigenfunction associated with $-\lambda_j(\mathcal{T}_x(\delta))$.
	\item\label{exa} There holds
	$$\lambda_j(\mathcal{T}_x(\delta))=\frac{(2j-1)\pi}{4}+ \frac{2}{(2j-1)\pi}\delta+\left(\frac{-16}{(2j-1)^3\pi^3}+ \frac{2}{(2j-1)\pi}\right) \delta^2 +\mathcal{O}(\delta^3),$$
	\item \label{iii}For $j\in \mathbb{N}$, a normalized eigenfunction associated with $\pm\lambda_j(\mathcal{T}_x(\delta))$ verifies
	\begin{align}
		\phi_{j,\delta}^\pm(t) := \phi_{j}^\pm(t) + \mathcal{O}(\delta),\label{exp}
\end{align}
	where $\phi_j^+$ is defined as
	\begin{align}
		\phi_j^+(t) = \frac12 \cos\left( \frac{(2j-1)\pi}{4}(t+1)\right) \begin{pmatrix}1\\-i(x_1 + ix_2)\end{pmatrix} + \frac12 \sin\left( \frac{(2j-1)\pi}{4}(t+1)\right) \begin{pmatrix}1\\i(x_1+ix_2)\end{pmatrix}\label{phi}
	\end{align}
	and $\phi_j^- = \sigma_1 \overline{\phi_j^+}$. Here, the remainder is understood in the $L^\infty$-norm. 
	\item\label{itm:pt4} For all $f \in \text{dom}(\mathcal{T}_x(\delta))$ there holds
		\[
			\|\mathcal{T}_x(\delta) f\|_{L^2((-1,1),\mathbb{C}^2)}^2 = \|f'\|_{L^2((-1,1),\mathbb{C}^2)}^2 + \delta^2 \|f\|_{L^2((-1,1),\mathbb{C}^2)}^2 + \delta \big(|f(1)|^2 + |f(-1)|^2\big).
		\]
\end{enumerate}
\label{prop:transope}
\end{prop}
The proof of the above proposition is similar to the proof of \cite[Prop.10]{th}, nevertheless we give it for the sake of completeness
\begin{proof}
Since the multiplication operator by $\sigma_3$ is bounded and self-adjoint in $L^2((-1,1),\mathbb{C}^2)$ the operator $\mathcal{T}_{x}(\delta)$ is self-adjoint if and only if  $\mathcal{T}_{x}(0)$ is self-adjoint. We can easily show that the operator $\mathcal{T}_{x}(0)$ is symmetric using an integration by parts.

For all $v\in\mathcal{D}:=C^{\infty}_0((-1,1),\mathbb{C}^2)$ and $u\in\text{dom}\mathcal{T}_{x}(0)^{*}$, there holds

\begin{align*}
\left\langle \mathcal{T}^*_x(0) u,v \right\rangle_{L^2((-1,1),\mathbb{C}^2)}=\left\langle  u, \mathcal{T}_x(0)v \right\rangle_{L^2((-1,1),\mathbb{C}^2)}&=\left\langle  u, -i\sigma\cdot v' \right\rangle_{L^2((-1,1),\mathbb{C}^2)}\\&=	\left\langle  u,\overline{ -i\sigma\cdot v'} \right\rangle_{\mathcal{D}',\mathcal{D}}\\&=	\left\langle -i\sigma\cdot u',\overline{v} \right\rangle_{\mathcal{D}',\mathcal{D}}\\&=	\left\langle \mathcal{T}_x(0)^* u,\overline{v}  \right\rangle_{\mathcal{D}',\mathcal{D}},
\end{align*}	
	where there $\left\langle .,.  \right\rangle_{\mathcal{D}',\mathcal{D}}$ is the duality bracket of distributions. In particular, we know that $\mathcal{T}_x(0)^*=-i\sigma\cdot u'\in L^2((-1,1),\mathbb{C}^2)$ thus we get $u\in H^1((-1,1),\mathbb{C}^2)$. Next, an integration by parts gives for $u=(u_1,u_2)^{\top}\in \text{dom}(\mathcal{T}_{x}(0)^*)$ and $v=(v_1,v_2)^{\top}\in \text{dom}(\mathcal{T}_{x}(0))$
\begin{multline*}
	\left\langle\mathcal{T}_{x}(0)^* u,v   \right\rangle _{L^2((-1,1),\mathbb{C}^2)}=\left\langle u,\mathcal{T}_{x}(0)v   \right\rangle _{L^2((-1,1)\mathbb{C}^2)}+(\overline{i(x_1+ix_2)}u_2(1)-u_1(1))\overline{v}_1(1)\\+(\overline{-i(x_1+ix_2)}u_2(-1)-u_1(-1))\overline{v}_1(-1) .
\end{multline*}
Recall that $v\in \text{dom}(\mathcal{T}_{x}(0))$, so necessarily there holds
$$ 0=(\overline{i(x_1+ix_2)}u_2(1)-u_1(1))\overline{v}_1(1)+(\overline{-i(x_1+ix_2)}u_2(-1)-u_1(-1))\overline{v}_1(-1).  $$
Since this is true for all $v\in \text{dom}(\mathcal{T}_{x}(0))$, we get $u_2(\pm1)=\pm i(x_1+ix_2) u_1(\pm 1)$ which yields $\mathcal{T}_{x}(0)^*=\mathcal{T}_{x}(0).$\\ The compactness of the resolvent of $\mathcal{T}_{x}(0)$ is a consequence of the graph theorem and the compactness of the embedding of $H^1((-1,1), \mathbb{C}^2)$ in $L^2((-1,1),\mathbb{C}^2)$. \\
Let $u\in\mathcal{T}_{x}(\delta)$, we have 
\begin{multline*}
	\| \mathcal{T}_{x}(\delta)u  \|^2_{L^2((-1,1),\mathbb{C}^2)}=\| -i\sigma\cdot x u' \|^2_{L^2((-1,1),\mathbb{C}^2)}+\delta^2\| u \|^2_{L^2((-1,1),\mathbb{C}^2)}\\+2\delta\Re(\left\langle   -i\sigma \cdot x u',\sigma_3u \right\rangle_{L^2((-1,1),\mathbb{C}^2)} )\geq \delta^2\| u  \|^2_{L^2((-1,1),\mathbb{C}^2)},
\end{multline*}
since an integration by parts and the boundary conditions gives us

\begin{align*}2\Re\big(\left\langle -i\sigma \cdot x u',\sigma_3u \right\rangle_{L^2((-1,1),\mathbb{C}^2)}\big)&=|u(1)|^2+|u(-1)|^2.
\end{align*} 
Note that using $\| -i\sigma\cdot x u' \|_{L^2((-1,1),\mathbb{C}^2)} = \| u' \|_{L^2((-1,1),\mathbb{C}^2)}$ proves Point \ref{itm:pt4}. By the min-max principle we deduce that if  $\lambda\in Sp(\mathcal{T}_{x}(\delta))$, then $|\lambda| \geq \delta$. Actually, the inequality is strict. Indeed, if $u$ is an eigenfunction of $\mathcal{T}_{x}(\delta)$ associated with $\lambda$ such that $|\lambda|=\delta$ we necessarily obtain that $u$ is a constant fuction on $(-1,1)$ verifying the boundary conditions. This not possible unless $u$ is $0$. We have shown that $Sp(\mathcal{T}_{x}(\delta))\cap[-\delta, \delta]=\emptyset$. \\Now, we compute its spectrum. To do so we pick  $\lambda\in Sp(\mathcal{T}_{x}(\delta))$ and chose a corresponding eigenfunction $u=(u_1,u_2)^{\top}\in \text{dom}(\mathcal{T}_{x}(\delta))$. We get  \begin{eqnarray}-u''_1=E^2u_1,\quad E^2=\lambda^2-\delta^2. \label{sol}\end{eqnarray}
  Taking into acount that $\delta\geq0$ and $E>0$, we obtain
$$u_1(t)=A\cos( E(t+1))+B\sin( E(t+1)), $$
where $A$, $B\in\mathbb{C}$ and since $\lambda+\delta\neq0$ there holds
$$u_2(t)=\frac{-i(x_1+ix_2) E}{\lambda+\delta}B \cos( E(t+1))+\frac{i(x_1+ix_2) E}{\lambda+\delta}A \sin( E(t+1)).$$
Using the boundary conditions at $t=-1$ we get
$$\sqrt{\lambda+\delta}A-\sqrt{\lambda-\delta}B=0.    $$
The boundary condition at $t=1$ gives
$$A(-\sqrt{\lambda+\delta}\cos2E+\sqrt{\lambda-\delta}\sin2E)-B(\sqrt{\lambda-\delta}\cos2E+\sqrt{\lambda+\delta}\sin2E)=0.  $$
To obtain a non-zero eigenfunction $u$, there must hold
\begin{align*}
	\left|  \begin{array}{cc}\sqrt{\lambda+\delta}&-\sqrt{\lambda-\delta} \\ \sqrt{\lambda-\delta}\sin2E-\sqrt{\lambda+\delta}\cos2E&-\sqrt{\lambda-\delta}\cos2E-\sqrt{\lambda+\delta}\sin2E  \end{array} \right|  =0
	\end{align*}
which yields the following implicit equation 
\begin{align}\delta\sin( 2E)+E\cos(2E)=0. \label{eq}
\end{align} 
In particular, this implies that the spectrum of $\mathcal{T}_x(\delta)$ is symmetric with respect to the origin. If $\delta=0$ then we have $E=|\lambda|=k\frac{\pi}{4}$, where $k\in\mathbb{N}$. If $\delta>0$, then we have $\cos(2E)\neq0$ and we get 

\begin{align}
\tan(2E)+\frac{E}{\delta}=0. \label{tanE}
\end{align}

Let us consider the following function 
\begin{align*}
	f_0:I_0=[0,\frac{\pi}{2})&\rightarrow\mathbb{R}\\x&\mapsto\tan(x)+\frac{x}{2\delta},
\end{align*}
we easily show that  $x=0$ is the unique solution to $f_0(x)=0$.

In the same way we define for $p\in \mathbb{N}$  the following function 

\begin{align*}
	f_p:I_p=((2p-1)\frac{\pi}{2},(2p+1)\frac{\pi}{2})&\rightarrow\mathbb{R}\\x&\mapsto\tan(x)+\frac{x}{2\delta}.
\end{align*}

Notice that $f'_p>0$ and 
\begin{align*}
	\lim\limits_{x\rightarrow(2p-1)\frac{\pi}{2}^+}f_p(x)=-\infty,\quad f_p(p\pi)=p\frac{\pi}{2\delta}>0.
\end{align*}
Thus, $f_p(x)=0$ admits a unique solution $x_p$ that verifies $x_p\in ((2p-1)\frac{\pi}{2},p\pi)$.
Using (\ref{tanE}) and (\ref{sol}), we obtain that $\lambda_p(T_x(\delta))^2 := \frac{x_p^2}4 + \delta^2$ is the $p$-th positive eigenvalue of  $T_x(\delta)$. In particular, we get \eqref{eqn:encalj}.

Now, consider the following $C^{\infty}$ function 
\begin{eqnarray}
F:\left\lbrace \begin{aligned} \mathbb{R}\times\mathbb{R}&\rightarrow \mathbb{R}\\(\mu,\delta)&\mapsto2\delta\sin(\mu)+\mu\cos\mu. \end{aligned}\right.
\end{eqnarray}
Since for $p\in \mathbb{N}$ we have $F(\frac{(2p-1)\pi}{2},0)=0$ and $\partial_{\mu}F(\frac{(2p-1)\pi}{2},0)=(-1)^{(p+1)}\frac{(2p-1)\pi}{2}$ one gets by the implicit function theorem that there exists $\theta_1,\theta_2>0$
and a $C^{\infty}$ function $\mu_p:(-\theta_1,\theta_1)\rightarrow(\frac{(2p-1)\pi}{2}-\theta_2,\frac{(2p-1)\pi}{2}+\theta_2)$ verifying $\mu_p(0)=(2p-1)\frac{\pi}{2}$ and such that for all $|\delta|<\theta_1$ we have $F(\mu_p(\delta),\delta)=0$. Moreover, when $\delta\rightarrow0$ we obtain
  
$$\mu_p(\delta)=\mu_p(0)+\mu_p'(0)\delta+\frac{\mu_p''(0)}2\delta^2+\mathcal{O}(\delta^3)=(2p-1)\frac{\pi}{2}+\frac{4}{(2p-1)\pi}\delta-\frac{32}{(2p-1)^3\pi^3}\delta^2+\mathcal{O}(\delta^3).$$
Hence, when $\delta\rightarrow 0$, we get
$$\lambda_p(\mathcal{T}_x(\delta))=\frac{(2p-1)\pi}{4}+ \frac{2}{(2p-1)\pi}\delta+\left(\frac{-16}{(2p-1)^3\pi^3}+ \frac{2}{(2p-1)\pi}\right) \delta^2 +\mathcal{O}(\delta^3),$$ this proves \ref{exa}. Finally we prove \ref{iii}.
For $p\in\mathbb{N}$, let $\phi^+_{p,\delta}$ be a corresponding normalized eigenfunction to $\lambda_p(\mathcal{T}_x(\delta))$ given by 
\begin{align}
	\phi^+_{p,\delta}(t)=\frac{1}{N_{p,\delta}^{\frac{1}{2}}}\cos(E_p(t+1))\left( \begin{array}c 1\\-i(x_1+ix_2) \end{array} \right)+\frac{1}{N_{p,\delta}^{\frac{1}{2}}}\sin(E_p(t+1))\left( \begin{array}c \frac{\lambda_p(\mathcal{T}_x(\delta))+\delta}{E_p}\\i(x_1+ix_2)\frac{E_p}{\lambda_p(\mathcal{T}_x(\delta))+\delta} \end{array} \right),\label{so}
\end{align}
where we have set $E_p=\sqrt{\lambda_p(\mathcal{T}_x(\delta))^2-\delta^2}$ and 
\begin{multline*} 
	N_{p,\delta}=2(1+\frac{1}{4E_p}\sin(4E_p))+(1-\frac{1}{4E_p}\sin(4E_p))\left(   \left(  \frac{\lambda_p(\mathcal{T}_x(\delta))+\delta}{E_p}  \right)^2+\left( \frac{E_p}{\lambda_p(\mathcal{T}_x(\delta))+\delta}   \right)^2   \right)\\+\frac{1}{2E_p}(1-\cos4E_p)\left(\frac{\lambda_p(\mathcal{T}_x(\delta))+\delta}{E_p}-\frac{E_p}{\lambda_p(\mathcal{T}_x(\delta))+\delta}  \right). 
\end{multline*}
Remark that when $\delta\rightarrow 0$
 we get
 
 \begin{align}
 \frac{1}{N_{p,\delta}^{\frac{1}{2}}}=\frac{1}{2}+\mathcal{O}(\delta),\quad \frac{\lambda_p(\mathcal{T}_x(\delta))+\delta}{E_p}=1+\mathcal{O}(\delta)\quad\text{ and }\frac{E_p}{\lambda_p(\mathcal{T}_x(\delta))+\delta}=1+\mathcal{O}(\delta)
 \label{so1}\end{align}
 and also that we have for all $t\in(-1,1)$
 \begin{align*}
 	\cos(E_p(t+1)) - \cos((2p-1)\frac\pi4(t+1)) &= \cos(\frac{\pi}4(2p-1)(t+1))(\cos(c_\delta(t+1)) - 1) \\&\qquad- \sin(\frac\pi4(2p-1)(t+1))\sin(c_\delta(t+1)),
 \end{align*}
 for some constant $c_\delta = \mathcal{O}(\delta)$ when $\delta \to 0$. By the mean value theorem applied to the cosine and sine functions we get
 \[
 	|\cos(E_p(t+1)) - \cos((2p-1)\frac\pi4(t+1))|\leq 2 |c_\delta (t+1)| \leq 4 |c_\delta|.
 \]
 Hence, $\cos(E_p(t+1)) = \cos((2p-1)\frac\pi4(t+1)) + \mathcal{O}(\delta)$ with a uniform remainder. Similarly, one would get $\sin(E_p(t+1)) = \sin((2p-1)\frac\pi4(t+1)) + \mathcal{O}(\delta)$. Hence, from this and from $(\ref{so})$ and $(\ref{so1})$ we get   
\[
	\phi^+_{p,\delta}(t) = \phi^+_p(t) + \mathcal{O}(\delta),
\] with a uniform remainder and we obtain $(\ref{exp})$.

\end{proof}

\section{Auxiliary quadratic forms}
In this section we give a two-sided estimate of the quadratic form of the square of the operator $\mathcal{D}_{\Gamma}(\varepsilon)$.

Let us introduce the quadratic form associated with $\mathcal{D}_\Gamma(\varepsilon)^2$ defined as
\[
a[u]:= \| D_{\Gamma}(\varepsilon)u \|^2,\quad \text{dom}(a):=\text{dom}(D_{\Gamma}(\varepsilon)).
\]We have the following proposition

\begin{prop}
	There exists $\varepsilon'>0$ and a constant $C>0$ such that for all $\varepsilon\in(0,\varepsilon')$ there exists a unitary map $U : L^2(\Omega_\varepsilon, \mathbb{C}^2) \to L^2(\mathbb{T}\times(-1,1),\mathbb{C}^2)$ such that for all $u \in \text{dom}(a)$ we have
	\[
	c^+[Uu] \leq a[u] \leq c^-[Uu],
	\]
where $c^\pm$ are the quadratic forms given by
\begin{multline*}
c^\pm[w] := (1\pm C \varepsilon)\int_{\mathbb{T}\times(-1,1)}|\partial_s w|^2dsdt - \int_{\mathbb{T}\times(-1,1)}\frac{\kappa^2}{4}|w|^2dsdt+(m^2\pm C\varepsilon) \|w\|^2 \\+ \frac{1}{\varepsilon^2}\Big(\int_\mathbb{T}\Big(\int_{-1}^1 |\partial_t w|^2 dt + m\varepsilon(|w(s,1)|^2 + |w(s,-1)|^2)\Big) ds\Big),
\\ \text{dom}(c^\pm)= U(\text{dom}(a))=\{u \in H^1(\mathbb{T}\times(-1,1),\mathbb{C}^2) : \forall s \in \mathbb{T}, \mp i\sigma_3\sigma\cdot\nu(s) u(s,\pm 1) = u(s,\pm1)\}.
\end{multline*}

 \end{prop}

\begin{proof}
By adapting the proof of \cite[Lemma 3.3, Prop. 3.5]{LOB}, for all $u\in \text{dom}(a)$, there holds
\begin{multline*}
	a[u] = \int_{\Omega_\varepsilon} |\nabla u|^2 dx + m^2 \|u\|^2 + \int_{0}^\ell\Big(m + \frac{\kappa(s)}{2(1+\varepsilon\kappa(s))}\Big)|u(\Phi_\varepsilon(s,-1))|^2(1+\varepsilon\kappa(s)) ds\\+ \int_0^\ell (m - \frac{\kappa(s)}{2(1-\varepsilon\kappa(s))})|u(\Phi_\varepsilon(s,1))|^2(1-\varepsilon \kappa(s)) ds.
\end{multline*}
Consider the unitary map $U_1$ defined as follows
\[
	U_1 : L^2(\Omega_\varepsilon) \to L^2(\mathbb{T}\times(-1,1),g_\varepsilon(s,t) ds dt),\quad (U_1 u)(s,t) := u (\Phi_\varepsilon(s,t)).
\]
where the metric $g_{\varepsilon}$ is given by $g_\varepsilon(s,t) := \varepsilon (1-\varepsilon t \kappa(s))$. In particular, for all $v \in U_1(\text{dom}(a))$ there holds
\begin{multline*}
	b[v]:=a[U_1^{-1} v] = \int_{\mathbb{T}\times(-1,1)}\Big(\frac1{(1-\varepsilon t\kappa(s))^2}|\partial_s v|^2 + \frac1{\varepsilon^2}|\partial_t v|^2\Big)g_\varepsilon(s,t) ds dt \\+ \frac1\varepsilon\int_0^\ell\Big((m + \frac{\kappa(s)}{2(1+\varepsilon\kappa(s))})|v(s,-1)|^2 {g_\varepsilon(s,-1)} \Big)ds\\+ \frac1\varepsilon\int_0^\ell\Big((m - \frac{\kappa(s)}{2(1-\varepsilon\kappa(s))})|v(s,1)|^2 {g_\varepsilon(s,1)} \Big)ds + m^2 \int_{\mathbb{T}\times(-1,1)}\Big(|v|^2g_\varepsilon(s,t)\Big)ds dt.
\end{multline*}
Now, consider the unitary map
\[
	U_2 : L^2(\mathbb{T}\times(-1,1),g_{\varepsilon}(s,t)ds dt) \to L^2(\mathbb{T}\times(-1,1)),\quad (U_2v)(s,t) := \sqrt{g_\varepsilon(s,t)}v.
\]
For $w = \sqrt{g_\varepsilon}v$, there holds
\[
	(\partial_t v)(s,t) = \partial_t (\frac1{\sqrt{g_\varepsilon(s,t)}}w(s,t)) = \frac{\varepsilon^2 \kappa(s)}{2g_\varepsilon(s,t)^{\frac32}}w(s,t) + \frac1{\sqrt{g_\varepsilon(s,t)}}\partial_t w(s,t)
\]
and
\[
	(\partial_s v)(s,t) = \partial_s (\frac1{\sqrt{g_\varepsilon(s,t)}}w(s,t)) = \frac{\varepsilon^2t \kappa'(s)}{2g_\varepsilon(s,t)^{\frac32}}w(s,t) + \frac1{\sqrt{g_\varepsilon(s,t)}}\partial_s w(s,t).
\]
In particular we get
\[
	\frac1{\varepsilon^2}|\partial_t v|^2 g_\varepsilon(s,t)= \frac1{\varepsilon^2}|\partial_t w |^2 + \frac{\kappa^2}{4 (1-\varepsilon t \kappa)^2}|w|^2 + \frac{\kappa}{2\varepsilon(1-\varepsilon t\kappa)}\partial_t(|w|^2)
\]
and
\[
 	\frac1{(1-\varepsilon t \kappa)^2} |\partial_s v|^2 g_\varepsilon(s,t) = \frac{1}{(1-\varepsilon t \kappa)^2}|\partial_s w|^2 + \frac{\varepsilon^2 t^2 (\kappa')^2}{4(1-\varepsilon t\kappa)^4}|w|^2 + \frac{\varepsilon t \kappa'}{2(1-\varepsilon t \kappa)^3}\partial_s(|w|^2).
\]
Remark that an integration by parts yields
\begin{multline*}
	\int_{-1}^1\frac{\kappa}{2\varepsilon(1-\varepsilon t \kappa)} \partial_t(|w|^2) dt = - \int_{-1}^1 \frac{\kappa^2}{2(1-\varepsilon t\kappa)^2}|w|^2 dt + \frac{\kappa}{2\varepsilon(1-\varepsilon \kappa)}|w(s,1)|^2\\  - \frac{\kappa}{2\varepsilon(1+\varepsilon \kappa)}|w(s,-1)|^2,\end{multline*}
as well as
\[
	\int_{\mathbb{T}}\frac{\varepsilon t \kappa'}{2(1-\varepsilon t \kappa)^3}\partial_s(|w|^2)ds = - \varepsilon\int_{\mathbb{T}}\frac{ t \kappa''}{2(1-\varepsilon t \kappa)^3}|w|^2ds - \varepsilon^2\int_{\mathbb{T}} \frac{3 t^2 (\kappa')^2}{2(1-\varepsilon t \kappa)^4}|w|^2  ds.
\]
Now, pick $w \in U_2(U_1\text{dom}(a))$ and set $v = U_2^{-1}w$. In particular, remark that it yields
\begin{multline*}
	c[w] := b[U_2^{-1}w] = \int_{\mathbb{T}\times(-1,1)}\Big(\frac{1}{(1-\varepsilon t \kappa)^2}|\partial_s w|^2 -\varepsilon \frac{t \kappa''}{2(1-\varepsilon t \kappa)^3}|w|^2 - \varepsilon^2\frac54\frac{t^2(\kappa')^2}{(1-\varepsilon t \kappa)^4} |w|^2\Big) ds dt\\ - \int_{\mathbb{T}\times(-1,1)}\Big(\frac{\kappa^2}{4(1-\varepsilon t \kappa)^2}|w|^2\Big) ds dt +\int_{\mathbb{T}\times (-1,1)} \Big(\frac1{\varepsilon^2}|\partial_t w|^2\Big) ds dt\\ + \frac{m}\varepsilon\int_{\mathbb{T}} \Big(|w(s,-1)|^2 + |w(s,1)|^2\Big) ds + m^2 \| w\|^2.
\end{multline*}
Now, using that $\kappa,\kappa'$ and $\kappa''$ are in $L^\infty(\mathbb{T})$, $|t|<1$ and taking $U=U_2U_1$ we get the proposition. 
\end{proof}

\section{Proof of Theorem \ref{thm:1}}

In this section we will prove the main theorem by deriving an upper and a lower bound for the eigenvalues of the square of  $\mathcal{D}_{\Gamma}(\varepsilon)$.

\subsection{Upper bound} The aim of this paragraph is to prove the following proposition.
\begin{prop}\label{up}
Let $j \in \mathbb{N}$, there exists $\varepsilon_1 > 0$ such that for all $\varepsilon \in (0,\varepsilon_1)$ 
\[
	\mu_j(a) \leq  \frac{\lambda_1(\mathcal{T}_\nu(m\varepsilon))^2}{\varepsilon^2}+ \mu_j(q^{1D}) + \mathcal{O}(\varepsilon).
\]
\end{prop}

In order to give the proof we need the folowing lemma
\begin{lem}\label{cd}
Let $\phi_{1,m\varepsilon}^\pm$ be the eigenfunctions of the transverse operator $\mathcal{T}_{\nu} (m\varepsilon)$ given in \ref{iii} of Proposition \ref{prop:transope}. Then, we have 
\begin{align}
\partial_s	\phi_{1,m\varepsilon}^\pm(t) := \partial_s\phi_{1}^\pm(t) + \mathcal{O}(\varepsilon),\nonumber\\
\partial_s^2	\phi_{1,m\varepsilon}^\pm(t) := \partial_s^2\phi_{1}^\pm(t) + \mathcal{O}(\varepsilon)\label{exp2}
\end{align}
where
\begin{align}
\partial_s\phi_1^+(t) = \frac12 \cos(\frac\pi4(t+1))\begin{pmatrix}0\\\kappa n\end{pmatrix} + \frac12\sin(\frac\pi4(t+1))\begin{pmatrix}0\\-\kappa n\end{pmatrix}\label{phi}
\end{align}
and
\[
	\partial_s^2\phi_1^+(t) = \frac12 \cos(\frac\pi4(t+1))\begin{pmatrix}0\\\partial_s(\kappa n)\end{pmatrix} + \frac12\sin(\frac\pi4(t+1))\begin{pmatrix}0\\-\partial_s(\kappa n)\end{pmatrix}\label{phip}
\]
and for $j \in \{1,2\}$, $\partial_s^j\phi_1^- = \sigma_1 \overline{\partial_s^j\phi_1^+}$ with $n=\nu_1+i\nu_2$. Here, the remainder is understood in the $L^\infty$-norm. 
\end{lem}

\begin{proof}
After remarking that for all $t\in(-1,1)$ and $j \in \{1,2\}$, we have
\[
	\partial_s^j\Phi_{1,m\varepsilon}^+(t) = \frac1{N_{1,m\varepsilon}^\frac12}\cos(E_1(m\varepsilon)(t+1)) \begin{pmatrix}0 \\ \partial_s^{j-1}(\kappa n)\end{pmatrix} + \frac1{N_{1,m\varepsilon}^\frac12}\sin(E_1(m\varepsilon)(t+1)) \begin{pmatrix}0\\ - \partial_s^{j-1}(\kappa n) \frac{E_1(m\varepsilon)}{\lambda_1(\mathcal{T}_\nu(m\varepsilon)) + m\varepsilon}\end{pmatrix}
\]
from this point the proof of Lemma \ref{cd} goes along the same lines as the proof of Point \ref{iii}~Proposition \ref{tran}.
%
\end{proof}
\begin{myproof}{Proposition} {\ref{up} }
Let $f^\pm \in H^1(\mathbb{T})$ and consider the trial-function $u = f^+\phi_{1,m\varepsilon}^+ + f^- \phi_{1,m\varepsilon}^-$. One notices that $u \in \text{dom}(c^+)$ and by definition of $u$ there holds
\[
	c^+[u] = (1+ C \varepsilon)\int_{\mathbb{T}\times(-1,1)}|\partial_s u|^2 dsdt - \int_{\mathbb{T}\times(-1,1)}\frac{\kappa^2}{4}|u|^2 + \frac{1}{\varepsilon^2}\lambda_1(\mathcal{T}_\nu(m\varepsilon))^2\|u\|^2 +C\varepsilon \|u\|^2.
\]
We remark that thanks to the orthogonality of $\phi_{1,m\varepsilon}^+$ and $\phi_{1,m\varepsilon}^-$ in $L^2((-1,1),\mathbb{C}^2)$ we have
\[
	\|u\|^2 = \int_{\mathbb{T}}\Big(|f^+|^2 + |f^-|^2\Big) ds = \|f\|^2,\quad \int_{\mathbb{T}\times(-1,1)}\Big(\frac{\kappa^2}{4}|u|^2\Big)dsdt= \int_{\mathbb{T}}\frac{\kappa^2}{4} |f|^2 ds,
\]
where we have set $f = \begin{pmatrix}f^+\\f^-\end{pmatrix}\in H^1(\mathbb{T},\mathbb{C}^2)$.
Moreover, there holds
\begin{multline}
	|\partial_s u|^2 = |(f^+)'\phi_{1,m\varepsilon}^+ + (f^-)'\phi_{1,m\varepsilon}^-|^2 + |f^+\partial_s(\phi_{1,m\varepsilon}^+) + f^-\partial_s(\phi_{1,m\varepsilon}^-)|^2 \\+ 2 \Re\Big(\langle(f^+)'\phi_{1,m\varepsilon}^+ + (f^-)'\phi_{1,m\varepsilon}^-,f^+\partial_s(\phi_{1,m\varepsilon}^+) + f^-\partial_s(\phi_{1,m\varepsilon}^-)\rangle\Big).
\label{fq}\end{multline}
In particular, we have
\begin{align}
	\int_{-1}^1\Big(|(f^+)'\phi_{1,m\varepsilon}^+ + (f^-)'\phi_{1,m\varepsilon}^-|^2\Big) dt = |f'|^2.\label{fc}
\end{align}
By Lemma $\ref{cd}$ we have
\begin{multline*}
	|f^+\partial_s(\phi_{1,m\varepsilon}^+) + f^-\partial_s(\phi_{1,m\varepsilon}^-)|^2 = \frac{\kappa^2}4 \Big(\cos(\frac{\pi}4(t+1)) - \sin(\frac{\pi}4(t+1))\Big)^2|f|^2 + |f|^2\mathcal{O}(\varepsilon)\\+2\Re (f^+\overline{f^-})\mathcal{O}(\varepsilon)
\end{multline*}
and
\begin{align}
	\int_{-1}^1\Big(|f^+\partial_s(\phi_{1,m\varepsilon}^+) + f^-\partial_s(\phi_{1,m\varepsilon}^-)|^2\Big) dt = \frac{\kappa^2}4\big(2-\frac4\pi\big)|f|^2 + |f|^2\mathcal{O}(\varepsilon)+2\Re (f^+\overline{f^-})\mathcal{O}(\varepsilon),\label{sc}
\end{align}
where all the remainder are uniform in the variable $t$.
Moreover, we obtain
\begin{multline*}
\langle(f^+)'\phi_{1,m\varepsilon}^+ + (f^-)'\phi_{1,m\varepsilon}^-,f^+\partial_s(\phi_{1,m\varepsilon}^+) + f^-\partial_s(\phi_{1,m\varepsilon}^-)\rangle = (f^+)'\overline{f^+} \langle \phi_{1,m\varepsilon}^+,\partial_s \phi_{1,m\varepsilon}^+\rangle + (f^+)' \overline{f^-}\langle\phi_{1,m\varepsilon}^+,\partial_s \phi_{1,m\varepsilon}^-\rangle\\ + (f^-)'\overline{f^+}\langle \phi_{1,m\varepsilon}^-,\partial_s \phi_{1,m\varepsilon}^+\rangle + (f^-)'\overline{f^-}\langle \phi_{1,m\varepsilon}^-,\partial_s \phi_{1,m\varepsilon}^-\rangle.
\end{multline*}
But remark that there holds
\[
	\int_{-1}^1\Big(\langle\phi_{1,m\varepsilon}^+,\partial_s \phi_{1,m\varepsilon}^+\rangle\Big)dt = \frac{i\kappa}4(\frac4\pi-2) + \mathcal{O}(\varepsilon)
\]
as well as
\[
	\int_{-1}^1\Big(\langle\phi_{1,m\varepsilon}^+,\partial_s \phi_{1,m\varepsilon}^-\rangle\Big)dt = \mathcal{O}(\varepsilon).
\]
Using that $\phi^-_{1,m\varepsilon}=\sigma_1\overline{\phi^+_{1,m\varepsilon}}$, we get that
\[
	\int_{-1}^1\Big(\langle\phi_{1,m\varepsilon}^-,\partial_s \phi_{1,m\varepsilon}^-\rangle\Big)dt = \frac{i\kappa}4(2-\frac4\pi) + \mathcal{O}(\varepsilon)
\]
and
\[
	\int_{-1}^1\Big(\langle\phi_{1,m\varepsilon}^-,\partial_s \phi_{1,m\varepsilon}^+\rangle\Big)dt = \mathcal{O}(\varepsilon).
\]
Hence, we obtain
\begin{multline}\int_{-1}^1  \langle(f^+)'\phi_{1,m\varepsilon}^+ + (f^-)'\phi_{1,m\varepsilon}^-,f^+\partial_s(\phi_{1,m\varepsilon}^+) + f^-\partial_s(\phi_{1,m\varepsilon}^-)\rangle\\= (f^+)'\overline{f^+}\frac{i\kappa}{4} (\frac4\pi-2)+(f^-)'\overline{f^-}\frac{i\kappa}{4} (2-\frac4\pi) + \mathcal{O}(\varepsilon)\big((f^+)'\overline{f^+} + (f^+)'\overline{f^-} + (f^-)'\overline{f^+} + (f^-)' \overline{f^-}\big)\\=\left\langle f', \frac{i\kappa}{4}(2-\frac4\pi)\sigma_3f \right\rangle+ \mathcal{O}(\varepsilon)\big((f^+)'\overline{f^+} + (f^+)'\overline{f^-} + (f^-)'\overline{f^+} + (f^-)' \overline{f^-}\big).  \label{eqn:doubprod} 
 \end{multline}
Substituting $(\ref{fc})$, $(\ref{sc})$ and the above equality in $(\ref{fq})$ we get, for some constant $c' > 0$
\begin{equation}
	\int_{-1}^1 |\partial_s u|^2 dt \leq |f' + i \kappa(\frac12 - \frac1\pi)\sigma_3 f|^2 + \kappa^2 \Big(\frac12 - \frac1\pi\Big)^2 |f|^2 + c' \varepsilon |f'|^2 + c'\varepsilon |f|^2.
	\label{eqn:ubdsu}
\end{equation}
Now, remark that
\begin{align}
	|f'|^2 &= |f' + i \kappa(\frac12 - \frac1\pi)\sigma_3 f - i\kappa(\frac12 - \frac1\pi)\sigma_3 f |^2\nonumber\\& = |f' + i \kappa(\frac12 - \frac1\pi)\sigma_3 f|^2 + \kappa^2\big(\frac12-\frac1\pi\big)^2 |f|^2 - 2 \Re\Big(\langle f' + i \kappa(\frac12 - \frac1\pi)\sigma_3 f, i\kappa(\frac12 - \frac1\pi)\sigma_3 f\rangle\Big)\nonumber\\
	& \leq  |f' + i \kappa(\frac12 - \frac1\pi)\sigma_3 f|^2 + c''|f|^2 + c''(|f' + i \kappa(\frac12 - \frac1\pi)\sigma_3 f|^2)\label{eqn:ubfpri},
\end{align}
for some constant $c'' > 0$, where we have used the Cauchy-Schwarz inequality to control the last term and the elementary identity $ab \leq \frac12 a^2 + \frac12b^2$ (for $a,b\in \mathbb{R}$).

%

Taking into account \eqref{eqn:ubdsu} and \eqref{eqn:ubfpri} we finally obtain for some new constants $k,k' > 0$ that
\[
	c^+[u] \leq (1+k\varepsilon)\int_{\mathbb{T}}\Big(|f' +i \kappa(\frac12-\frac1\pi)\sigma_3 f|^2-\frac{\kappa^2}{\pi^2}|f|^2\Big)ds + \big( \frac1{\varepsilon^2}\lambda_1(\mathcal{T}_\nu(m\varepsilon))^2 + k'\varepsilon)\big)\|f\|^2.
\]
Using the min-max principle one gets the expected result.
\end{myproof}

\subsection{Lower bound}

The goal of this paragraph is to prove the following proposition.

\begin{prop}\label{pr4}Let $j \in \mathbb{N}$, there exists $\varepsilon_1 > 0$ such that for all $\varepsilon \in (0,\varepsilon_1)$
\[
	\mu_j(a)\geq \frac{\lambda_1(\mathcal{T}_\nu(m\varepsilon))^2}{\varepsilon^2} + \mu_j(q^{1D}) + \mathcal{O}(\varepsilon).
\]
\end{prop}
In order to prove the above proposition we need the next lemma. 
\begin{lem}\label{lem1}Let $\Pi_{\varepsilon}$ be the orthogonal projectors initialy defined for $u \in L^2((-1,1),\mathbb{C}^2)$ by
	\[
	\Pi_{\varepsilon} u = \langle u, \phi_{1,m\varepsilon}^+\rangle \phi_{1,m\varepsilon}^+ + \langle u, \phi_{1,m\varepsilon}^-\rangle \phi_{1,m\varepsilon}^-,\quad \Pi_{\varepsilon}^{\perp} = Id - \Pi_{\varepsilon}.
	\]
Then, the map $[\partial_s, \Pi_{\varepsilon}]u:=\partial_s(\Pi_{\varepsilon} u)-\Pi_{\varepsilon}(\partial_s u)$ defined for $u\in C^1(\mathbb{T},L^2(-1,1))$ extends by density to a bounded operator from $H^q(\mathbb{T},L^2(-1,1))$ to  $H^q(\mathbb{T},L^2(-1,1))$ still denoted $[\partial_s,\Pi_{\varepsilon}]$ and whose norm remains uniformly bounded for $\varepsilon\rightarrow 0$ ( for $q \in \{0,1\}$). The same conclusion holds for $[\partial_s,\Pi_{\varepsilon}^{\perp}]=-[\partial_s,\Pi_{\varepsilon}].$
\end{lem}
\begin{proof}The first part is proven by noticing that for $u\in C^1(\mathbb{T},L^2(-1,1))$ we have
	\begin{align}\label{comu}
	[\partial_s, \Pi_{\varepsilon}]u=\langle u,\partial_s \phi_{1,m\varepsilon}^+\rangle\phi_{1,m\varepsilon}^+ + \langle u,\partial_s \phi_{1,m\varepsilon}^-\rangle\phi_{1,m\varepsilon}^- + \left\langle u, \phi_{1,m\varepsilon}^+  \right\rangle \partial_s\phi_{1,m\varepsilon}^++\left\langle u, \phi_{1,m\varepsilon}^-  \right\rangle \partial_s\phi_{1,m\varepsilon}^-
	\end{align}
	 and 
\begin{align*}\|[\partial_s, \Pi_{\varepsilon}]u\|&\leq \|\langle u,\partial_s \phi_{1,m\varepsilon}^+\rangle\phi_{1,m\varepsilon}^+\| + \|\langle u,\partial_s \phi_{1,m\varepsilon}^-\rangle\phi_{1,m\varepsilon}^-\| + \|\left\langle u, \phi_{1,m\varepsilon}^+  \right\rangle \partial_s\phi_{1,m\varepsilon}^+\|\\& \qquad\qquad+ \|\left\langle u, \phi_{1,m\varepsilon}^-  \right\rangle \partial_s\phi_{1,m\varepsilon}^-\|\\&\leq \Big(\int_{\mathbb{T}} |\langle u,\partial_s \phi_{1,m\varepsilon}^+\rangle|^2ds\Big)^{\frac12} + \Big(\int_{\mathbb{T}} |\langle u,\partial_s \phi_{1,m\varepsilon}^-\rangle|^2ds\Big)^{\frac12} + \Big(\int_{\mathbb{T}} |\langle u,\phi_{1,m\varepsilon}^+\rangle|^2 \big(\int_{-1}^1 |\partial_s \phi_{1,m\varepsilon}^+|^2 dt \big)ds\Big)^{\frac12}\\
&\qquad\qquad + \Big(\int_{\mathbb{T}} |\langle u,\phi_{1,m\varepsilon}^-\rangle|^2 \big(\int_{-1}^1 |\partial_s \phi_{1,m\varepsilon}^-|^2 dt \big)ds\Big)^{\frac12}\\
&\leq \Big(\int_{\mathbb{T}}\big(\int_{-1}^1|u|^2dt\big) \big(\int_{-1}^1|\partial_s \phi_{1,m\varepsilon}^+|^2dt\big) ds\Big)^{\frac12} + \Big(\int_{\mathbb{T}}\big(\int_{-1}^1|u|^2dt\big) \big(\int_{-1}^1|\partial_s \phi_{1,m\varepsilon}^-|^2dt\big) ds\Big)^{\frac12}\\
&\qquad\qquad + \Big(\int_{\mathbb{T}} |\langle u,\phi_{1,m\varepsilon}^+\rangle|^2 \big(\int_{-1}^1 |\partial_s \phi_{1,m\varepsilon}^+|^2 dt \big)ds\Big)^{\frac12} + \Big(\int_{\mathbb{T}} |\langle u,\phi_{1,m\varepsilon}^-\rangle|^2 \big(\int_{-1}^1 |\partial_s \phi_{1,m\varepsilon}^-|^2 dt \big)ds\Big)^{\frac12}.
\end{align*}
Thanks to Lemma \ref{cd} and the explicit expression of $\partial_s\phi_{1,m\varepsilon}^\pm$ given in \eqref{phi} we know that there exists $c,k > 0$ such that for all $s \in \mathbb{T}$ we have for $\varepsilon$ sufficiently small
\[
	 \int_{-1}^1|\partial_s \phi_{1,m\varepsilon}^\pm|^2 dt \leq (c + k \varepsilon)^2.
\]
Here we have used that by hypothesis $\kappa \in L^\infty(\mathbb{T})$.
It yields
\begin{align*}
	\|[\partial_s, \Pi_{\varepsilon}]u\|&\leq 2(c+k\varepsilon) \|u\| + (c+k\varepsilon) \Big(\Big(\int_{\mathbb{T}} |\langle u,\phi_{1,m\varepsilon}^+\rangle|^2ds\Big)^{\frac12} + \Big(\int_{\mathbb{T}} |\langle u,\phi_{1,m\varepsilon}^-\rangle|^2ds\Big)^{\frac12}\Big)\\& \leq 4(c+k\varepsilon) \|u\|,
\end{align*}
where we have used that the last two-terms in the first line are controlled by the norm of the orthogonal projectors on $\phi_{1,m\varepsilon}^\pm$. Now, recall that $\gamma$ is of class $C^4$ which implies that $\phi_{1,m\varepsilon}^\pm \in C^3(\mathbb{T},C^{\infty}(-1,1))$ and $\partial_s\phi_{1,m\varepsilon}^\pm \in C^2(\mathbb{T},C^{\infty}(-1,1))$. Taking this into account for all $u \in C^1(\mathbb{T},L^2((-1,1)))$ we have
\[
	\partial_s [\partial_s,\Pi_\varepsilon] u = [\partial_s,\Pi_\varepsilon] \partial_s u + \sum_{j = \pm} \Big(\langle u, \partial_s^2 \phi_{1,m\varepsilon}^j\rangle\phi_{1,m\varepsilon}^j + \langle u,\phi_{1,m\varepsilon}^j\rangle \partial_s^2 \phi_{1,m\varepsilon}^j + 2 \langle u, \partial_s \phi_{1,m\varepsilon}^j\rangle \partial_s\phi_{1,m\varepsilon}^j\Big)
\]
Proceeding as above and using once again Lemma \ref{cd}, we get that for $\varepsilon$ small enough there exists some constant $c > 0$ such that
\[
	\|\partial_s [\partial_s,\Pi_\varepsilon] u\| \leq c\big(\|\partial_s u\| + \|u\|\big).
\]
Hence, this extends by density to any function in $H^1(\mathbb{T},L^2((-1,1)))$ and the lemma is proved.
	\end{proof}
\begin{myproof} {Theorem}{\ref{pr4}} 
	Note that for all $u \in \text{dom}(c^-)$ there holds
	\begin{multline*}
		\int_{\mathbb{T}\times (-1,1)} |\partial_s u|^2 ds dt = \int_{\mathbb{T}\times (-1,1)} |\partial_s (\Pi_{\varepsilon} u + \Pi_{\varepsilon}^\perp u)|^2 ds dt = \|\partial_s \Pi_{\varepsilon} u\|^2 + \|\partial_s \Pi_{\varepsilon}^{\perp} u\|^2 + 2 \Re\Big(\langle\partial_s \Pi_{\varepsilon} u,\partial_s \Pi_{\varepsilon}^\perp u\rangle\Big).
	\end{multline*}
Let us focus on the last term in the last equality. There holds
\begin{multline*}
	\langle \partial_s \Pi_{\varepsilon} u,\partial_s \Pi^{\perp}_{\varepsilon} u \rangle = \langle \partial_s \Pi^2_{\varepsilon} u,\partial_s (\Pi_{\varepsilon}^\perp)^2 u \rangle = \langle ([\partial_s,\Pi_{\varepsilon}] +\Pi_{\varepsilon}\partial_s)\Pi_{\varepsilon} u, ([\partial_s,\Pi_{\varepsilon}^\perp] +\Pi_{\varepsilon}^\perp\partial_s)\Pi_{\varepsilon}^\perp u\rangle\\ = \langle[\partial_s,\Pi_{\varepsilon}] \Pi_{\varepsilon} u, [\partial_s,\Pi_{\varepsilon}^\perp]\Pi_{\varepsilon}^\perp u\rangle + \langle[\partial_s,\Pi_{\varepsilon}]\Pi_{\varepsilon} u, \Pi_{\varepsilon}^\perp \partial_s \Pi_{\varepsilon}^\perp u\rangle + \langle \Pi_{\varepsilon}\partial_s \Pi_{\varepsilon} u, [\partial_s,\Pi_{\varepsilon}^\perp]\Pi_{\varepsilon}^\perp u \rangle\\ := J_1 + J_2 + J_3.
\end{multline*}
Now, we deal with each term appearing in the right-hand side of the last expression.
Using Lemma \ref{lem1}, and the elementary
inequality $ab\leq\frac{a^2\varepsilon}{2}+\frac{b^2}{2\varepsilon}$, where $a,b\in\mathbb{R}$ and $\varepsilon > 0$, we obtain
\[
	|J_1| \leq \|[\partial_s, \Pi_{\varepsilon}] \Pi_{\varepsilon} u\| \|[\partial_s, \Pi_{\varepsilon}^\perp]\Pi_{\varepsilon}^\perp u\| \leq \frac{c_1}2\varepsilon \|\Pi_{\varepsilon} u\|^2 + \frac{c_1}{2\varepsilon}\|\Pi_{\varepsilon}^\perp u\|^2,
\]where $c_1>0 $ is a constant.
Regarding $J_3$, similarly one gets that there exists $c_2>0$ such that 
\[
	|J_3| \leq \|\Pi_{\varepsilon}\partial_s \Pi_{\varepsilon} u\| \|[\partial_s,\Pi_{\varepsilon}^\perp]\Pi_{\varepsilon}^\perp u\| \leq c_2\|\partial_s \Pi_{\varepsilon} u\| \|\Pi_{\varepsilon}^\perp u\| \leq \frac{c_2}2\varepsilon \|\partial_s \Pi_{\varepsilon} u\|^2 + \frac{c_2}{2\varepsilon}\|\Pi_{\varepsilon}^\perp u\|^2.
\]
We are left with the investigation of the term $J_2$. Note that there holds
\[
	J_2 = \langle \Pi_{\varepsilon}^\perp [\partial_s, \Pi_{\varepsilon}]\Pi_{\varepsilon} u,\partial_s \Pi_{\varepsilon}^\perp u\rangle.
\]
In particular, we get
\begin{equation}
	\Pi_{\varepsilon}^\perp [\partial_s, \Pi_{\varepsilon}]\Pi_{\varepsilon} u = \langle \Pi_{\varepsilon} u, \phi_{1,m\varepsilon}^+\rangle \Pi_{\varepsilon}^\perp\partial_s \phi_{1,m\varepsilon}^+ + \langle \Pi_{\varepsilon} u, \phi_{1,m\varepsilon}^-\rangle \Pi_{\varepsilon}^\perp\partial_s \phi_{1,m\varepsilon}^-
	\label{eqn:exprcom}
\end{equation}
and we note that $\Pi_{\varepsilon}^\perp [\partial_s, \Pi_{\varepsilon}]\Pi_{\varepsilon} u \in H^1(\mathbb{T}\times(-1,1),\mathbb{C}^2)$. In particular performing an integration by parts, one gets
\[
	\langle \Pi_{\varepsilon}^\perp [\partial_s, \Pi_{\varepsilon}]\Pi_{\varepsilon} u,\partial_s \Pi_{\varepsilon}^\perp u\rangle = - \langle \partial_s\Pi_{\varepsilon}^\perp [\partial_s, \Pi_{\varepsilon}]\Pi_{\varepsilon}u, \Pi_{\varepsilon}^\perp u\rangle,
\]
as the boundary term vanishes. Now, using \eqref{eqn:exprcom} we note that
\begin{multline}
	\partial_s\Pi_{\varepsilon}^\perp [\partial_s, \Pi_{\varepsilon}]\Pi_{\varepsilon} u = \langle \partial_s \Pi_{\varepsilon} u, \phi_{1,m\varepsilon}^+\rangle \Pi_{\varepsilon}^\perp \partial_s\phi_{1,m\varepsilon}^+ + \langle \partial_s \Pi_{\varepsilon} u, \phi_{1,m\varepsilon}^-\rangle \Pi_{\varepsilon}^\perp\partial_s \phi_{1,m\varepsilon}^-+ \langle  \Pi_{\varepsilon} u,\partial_s \phi_{1,m\varepsilon}^+\rangle \Pi_{\varepsilon}^\perp \partial_s\phi_{1,m\varepsilon}^+ \\+ \langle \Pi_{\varepsilon} u, \partial_s\phi_{1,m\varepsilon}^-\rangle \Pi_{\varepsilon}^\perp\partial_s \phi_{1,m\varepsilon}^- + \langle \Pi_{\varepsilon} u, \phi_{1,m\varepsilon}^+\rangle \partial_s \Pi_{\varepsilon}^\perp\partial_s \phi_{1,m\varepsilon}^+ \\+ \langle \Pi_{\varepsilon} u, \phi_{1,m\varepsilon}^-\rangle \partial_s \Pi_{\varepsilon}^\perp \partial_s\phi_{1,m\varepsilon}^-.
\end{multline}
By Lemma $\ref{lem1}$, there exists $c_3>0$ such that 
\[
	\|\partial_s\Pi_{\varepsilon}^\perp [\partial_s, \Pi_{\varepsilon}]\Pi_{\varepsilon} u\| \leq c_3(\|\partial_s \Pi_{\varepsilon} u\|+\|\Pi_{\varepsilon} u\|).
\]
Hence, for $J_2$, there holds
\begin{multline*}
	|J_2| \leq c_3(\|\partial_s \Pi_{\varepsilon}u\| + \|\Pi_{\varepsilon} u\|)\|\Pi_{\varepsilon}^\perp u\|\leq c	_		3\frac{\varepsilon}2 (\|\partial_s \Pi_{\varepsilon} u\|+\|\Pi_{\varepsilon} u\|)^2 + c_3\frac{1}{2\varepsilon}\|\Pi_{\varepsilon}^\perp u\|^2 \\\leq c_3 \varepsilon \|\partial_s \Pi_{\varepsilon} u\|^2 + c_3\varepsilon \|\Pi_{\varepsilon} u\|^2 + \frac{c_3}{2\varepsilon}\|\Pi_{\varepsilon}^\perp u\|^2.
\end{multline*}
 Hence, collecting all the estimates, one gets for some $c_4>0$
\[
	\int_{\mathbb{T}\times (-1,1)} |\partial_s u|^2 ds dt \geq (1-c_4\varepsilon) \|\partial_s \Pi_{\varepsilon} u\|^2 - c_4\varepsilon \|\Pi_{\varepsilon} u\|^2 - \frac{c_4}\varepsilon \|\Pi_{\varepsilon}^\perp u\|^2.
\]
Thus, we obtain
\begin{multline*}
	c^-[u] \geq (1-C\varepsilon)(1-c_4\varepsilon)\|\partial_s \Pi_{\varepsilon} u\|^2 - (1-C\varepsilon)c_4\varepsilon \|\Pi_{\varepsilon} u\|^2 - (1-C\varepsilon)\frac{c_4}\varepsilon \|\Pi_{\varepsilon}^\perp u\|^2 - \int_{\mathbb{T}\times (-1,1)} \frac{\kappa^2}{4}|\Pi_{\varepsilon} u|^2 ds dt \\-  \int_{\mathbb{T}\times (-1,1)} \frac{\kappa^2}{4}|\Pi_{\varepsilon}^\perp u|^2 ds dt + \frac1{\varepsilon^2} \lambda_1(\mathcal{T}_\nu(m\varepsilon))^2 \|\Pi_{\varepsilon} u\|^2 \\+ \frac1{\varepsilon^2} \lambda_2(\mathcal{T}_\nu(m\varepsilon))^2 \|\Pi_{\varepsilon}^\perp u\|^2 - C \varepsilon\|u\|^2\\\geq (1-c_5\varepsilon) \|\partial_s \Pi_{\varepsilon} u\|^2 - \int_{\mathbb{T}\times(-1,1)}\frac{\kappa^2}{4}|\Pi_{\varepsilon} u|^2 ds dt + \big(\frac{1}{\varepsilon^2} \lambda_1(\mathcal{T}_{\nu}(m\varepsilon))^2-c_5 \varepsilon\big)\|\Pi_{\varepsilon} u\|^2\\ + (\frac{1}{\varepsilon^2}\lambda_2(\mathcal{T}_{\nu}(m\varepsilon))^2 - \frac{c_5}\varepsilon)\|\Pi_{\varepsilon}^\perp u\|^2,
\end{multline*}
for some constant $c_5>0.$ Here we have used that $\mu_1(\mathcal{T}_{\nu}(m\varepsilon)^2) = \lambda_1(\mathcal{T}_{\nu}(m\varepsilon))^2$, that $\mu_3(\mathcal{T}_{\nu}(m\varepsilon)^2) = \lambda_2(\mathcal{T}_{\nu}(m\varepsilon))^2$ and that $\kappa \in L^\infty(\mathbb{T})$. Now we infer that
\begin{multline*}
	\int_{-1}^1 |\partial_s \Pi_{\varepsilon} u|^2 dt = \int_{-1}^1 |\partial_s \big(\langle u, \phi_{1,m\varepsilon}^+\rangle\big) \phi_{1,m\varepsilon}^+ +\partial_s\big(\langle u, \phi_{1,m\varepsilon}^-\rangle\big) \phi_{1,m\varepsilon}^- + \langle u, \phi_{1,m\varepsilon}^+\rangle \partial_s \phi_{1,m\varepsilon}^+ + \langle u, \phi_{1,m\varepsilon}^-\rangle \partial_s \phi_{1,m\varepsilon}^-|^2 dt\\
	 =\int_{-1}^1  |\partial_s \big(\langle u, \phi_{1,m\varepsilon}^+\rangle\big) \phi_{1,m\varepsilon}^+ +\partial_s\big(\langle u, \phi_{1,m\varepsilon}^-\rangle \big)\phi_{1,m\varepsilon}^-|^2dt + \int_{-1}^1 |\langle u, \phi_{1,m\varepsilon}^+\rangle \partial_s \phi_{1,m\varepsilon}^+ + \langle u, \phi_{1,m\varepsilon}^-\rangle \partial_s \phi_{1,m\varepsilon}^-|^2 dt \\\qquad + 2\Re\bigg(\int_{-1}^1 \langle \partial_s\big(\langle  u, \phi_{1,m\varepsilon}^+\rangle\big) \phi_{1,m\varepsilon}^+ + \partial_s\big(\langle u, \phi_{1,m\varepsilon}^-\rangle\big) \phi_{1,m\varepsilon}^-,\langle u, \phi_{1,m\varepsilon}^+\rangle \partial_s \phi_{1,m\varepsilon}^+ + \langle u, \phi_{1,m\varepsilon}^-\rangle \partial_s \phi_{1,m\varepsilon}^-\rangle dt\bigg).
\end{multline*}
Next, we set $f^\pm := \langle u,\phi_{1,m\varepsilon}^\pm \rangle$ and $f = \begin{pmatrix}f^+\\f^-\end{pmatrix}$. Remark that $|f|^2 = \int_{-1}^1 |\Pi_{\varepsilon} u|^2dt$. Proceeding in the same way as for the obtention of Equations \eqref{sc}, \eqref{eqn:doubprod}, \eqref{eqn:ubdsu} and \eqref{eqn:ubfpri} one obtains the existence $c_6>0$ such that
\[
	\int_{-1}^1 |\partial_s \Pi_{\varepsilon} u|^2 dt \geq (1-c_6\varepsilon)|f' + i\kappa(\frac12-\frac1\pi)\sigma_3 f|^2 + \kappa^2\big(\frac12-\frac1\pi\big)^2|f|^2 - c_6 \varepsilon |f|^2.
\]
Hence, there holds
\begin{align}\label{form16}
	c^-[u]\geq (1-c_7\varepsilon)q^{1D}[f] + \big(\frac1{\varepsilon^2}\lambda_1(\mathcal{T}_\nu(m\varepsilon))^2- c_7'\varepsilon\big)\|f\|^2 + \big(\frac1{\varepsilon^2}\lambda_2(\mathcal{T}_\nu(m\varepsilon))^2- \frac{c_7
	}{\varepsilon}\big)\|\Pi_{\varepsilon}^\perp u\|^2,
\end{align}
for some constants $c_7,c_7'>0$.\\

Now we define the quadratic form $Q$ with tensor product domain:
\begin{align*}
	Q(f,v)&:=(1-c_7\varepsilon)q^{1D}[f] + \big(\frac1{\varepsilon^2}\lambda_1(\mathcal{T}_\nu(m\varepsilon))^2 - c_7'\varepsilon\big)\|f\|^2 + \big(\frac1{\varepsilon^2}\lambda_2(\mathcal{T}_\nu(m\varepsilon))^2- \frac{c_7
	}{\varepsilon}\big)\| v\|^2,\\\text{dom}(Q)&:=H^1(\mathbb{T},\mathbb{C}^2)\times\Pi_{\varepsilon}^{\perp} L^2(\mathbb{T}\times(-1,1),\mathbb{C}^2).
\end{align*}
Next, we give the quadratic form $Q_1$ acting on  $\Pi_{\varepsilon} L^2(\mathbb{T}\times(-1,1),\mathbb{C}^2)\times\Pi_{\varepsilon}^{\perp}L^2(\mathbb{T}\times(-1,1),\mathbb{C}^2)$ by $Q(V_1v_1,v_2)=Q_1(v_1,v_2),$ where $V_1$ is the unitary map given by 
\begin{align*}V_1:\Pi_{\varepsilon}L^2(\mathbb{T}\times(-1,1),\mathbb{C}^2)&\rightarrow L^2(\mathbb{T},\mathbb{C}^2)\\u&\mapsto V_1u=\begin{pmatrix}\left\langle u,\phi^+_{1,m\varepsilon}   \right\rangle \\\left\langle u, \phi^-_{1,m\varepsilon}      \right\rangle \end{pmatrix}.
\end{align*}

Let $u\in\text{dom}(c^-)$, by $(\ref{form16})$ we get
\begin{align*}
c^-[u]\geq Q(V_1\Pi_{\varepsilon} u,\Pi_{\varepsilon}^{\perp}u)=Q_1(\Pi_{\varepsilon} u, \Pi_{\varepsilon}^{\perp}u).
\end{align*}
The above inequality takes the form $c^-[u]\geq Q(Vu)$, where $Vu=(V_1\Pi_{\varepsilon} u, \Pi_{\varepsilon}^{\perp} u)$. As $V$ is unitary we get by the min-max principle that $\mu_j(c^-)\geq\mu_j(Q)$ for all $j\in\mathbb{N}$. Remark that we have the representation $Q_1=q_1\oplus q_2$ where $q_1$ and $q_2$ are two quadratic forms acting on $\Pi_{\varepsilon} L^2(\mathbb{T}\times(-1,1),\mathbb{C}^2)$ and $\Pi_{\varepsilon}^{\perp} L^2(\mathbb{T}\times(-1,1),\mathbb{C}^2)$, respectively by \begin{align*}
q_1[u]=(1-c_7\varepsilon)q^{1D}[V_1u] + \big(\frac1{\varepsilon^2}\lambda_1(\mathcal{T}_\nu(m\varepsilon))^2 - c_7'\varepsilon\big)\|V_1u\|^2 
\end{align*}
and 
\begin{align*}
q_2[v]=\big(\frac1{\varepsilon^2}\lambda_2(\mathcal{T}_\nu(m\varepsilon))^2- \frac{c_7
}{\varepsilon}\big)\| v\|^2,
\end{align*}

On the other hand since $\lambda_2(\mathcal{T}_\nu(m\varepsilon))^2>\frac{9\pi^2}{16}$ by \ref{exa} of Proposition \ref{tran}, and $\lambda_1(\mathcal{T}_\nu(m\varepsilon))^2 \leq \frac{\pi^2}4 + (m\varepsilon)^2$ one gets that for $j\in\mathbb{N}$, there exists $\varepsilon_1>0$ such that for $\varepsilon\in(0,\varepsilon_1)$ we have
\begin{align*} 
\frac1{\varepsilon^2}\lambda_2(\mathcal{T}_\nu(m\varepsilon))^2- \frac{c_7
}{\varepsilon} \geq (1-c_7\varepsilon)\lambda_j(q^{1D}) + \frac1{\varepsilon^2}\lambda_1(\mathcal{T}_\nu(m\varepsilon))^2 - c_7'\varepsilon.\end{align*}
This implies that 
\begin{align*}
	\mu_j(c^-)\geq\mu_j(q_1).	
\end{align*}
This concludes the proof of this proposition.\end{myproof}\\
Finally, we are able to prove the main theorem.

\begin{myproof} {Theorem}{\ref{thm:1}} The proof  is obtained taking into account Propositions \ref{up} and \ref{pr4} and noting that because of the symmetry of the spectrum we have $E_j(\varepsilon) = \sqrt{\mu_{2j}(a)}$. Hence, there holds
\begin{align*}
	E_j(\varepsilon)^2 &= \frac{\lambda_1(\mathcal{T}_\nu(m\varepsilon))^2}{\varepsilon^2} + \mu_{2j}(q^{1D}) + \mathcal{O}(\varepsilon)\\
	& = \frac{\pi^2}{16\varepsilon^2} + \frac{m}{\varepsilon} - \frac4{\pi^2}m^2 + m^2 + \mu_{2j}(q^{1D}) + \mathcal{O}(\varepsilon).
\end{align*}
Hence, there holds
\begin{align*}
	E_j(\varepsilon) &= \frac{\pi}{4\varepsilon} \sqrt{1 + \frac{16}{\pi^2}m\varepsilon - \frac{64}{\pi^4}m^2\varepsilon^2 + \frac{16}{\pi^2}(m^2 + \mu_{2j}(q^{1D}))\varepsilon^2} + \mathcal{O}(\varepsilon^2)\\& = \frac{\pi}{4\varepsilon}\Big(1 + \frac{8}{\pi^2}m\varepsilon - \frac{64}{\pi^4}m^2\varepsilon^2 + \frac{8}{\pi^2}\big(m^2 + \mu_{2j}(q^{1D})\big)\varepsilon^2\Big) + \mathcal{O}(\varepsilon^2)\\
	& = \frac{\pi}{4\varepsilon} + \frac2\pi m - \frac{16}{\pi^3}m^2\varepsilon + \frac{2}\pi(m^2 + \mu_{2j}(q^{1D}))\varepsilon + \mathcal{O}(\varepsilon^2)
\end{align*}
which proves Theorem \ref{thm:1}.
\end{myproof}

	\end{document}